\documentclass[12pt]{article}
\usepackage{amssymb}
\usepackage{amsmath,amsthm}
\usepackage[latin1]{inputenc}
\usepackage[dvips]{graphicx}
\usepackage{hyperref}
\usepackage{enumerate}
\hypersetup{colorlinks=true}

\hypersetup{colorlinks=true, linkcolor=blue, citecolor=blue,urlcolor=blue}

 \newtheorem{remark}{Remark}

 \newtheorem{lemma}[remark]{Lemma}
 \newtheorem{theorem}[remark]{Theorem}
 \newtheorem{proposition}[remark]{Proposition}
 \newtheorem{corollary}[remark]{Corollary}
  \newtheorem{claim}[remark]{Claim}
  \newtheorem{conjecture}[remark]{Conjecture}

\newcommand{\pd}{\operatorname{pd}}
\newcommand{\ex}{\operatorname{ex}}

\title{On the  partition dimension of  unicyclic graphs}

\author{ Juan A. Rodr\'{\i}guez-Vel\'{a}zquez \footnote{\small e-mail:\mbox{\tt juanalberto.rodriguez\@@urv.cat}} $^1$, Ismael G. Yero\footnote{\small e-mail:\mbox{\tt ismael.gonzalez\@@uca.es}} $^2$
and Henning Fernau\footnote{\small e-mail:\mbox{\tt
fernau\@@uni-trier.de}} $^3$  \\
$^1${\small Departament d'Enginyeria Inform\`{a}tica i Matem\`{a}tiques
}\\
{\small Universitat Rovira i Virgili,  Av. Pa\"{\i}sos Catalans 26, 43007
Tarragona, Spain.}
\\
$^2${\small Departamento de Matem\'aticas, Escuela Polit\'ecnica Superior de Algeciras}\\
{\small Universidad de C\'adiz,} {\small
Av. Ram\'on Puyol s/n, 11202 Algeciras, Spain.}
\\
$^3${\small FB 4-Abteilung Informatik Universit\"{a}t Trier, 54286
Trier, Germany.}}

\begin{document}

\maketitle

\begin{abstract}
Given an ordered partition $\Pi =\{P_1,P_2, ...,P_t\}$ of the
vertex set $V$ of a connected graph $G=(V,E)$, the  \emph{partition representation} of a vertex
$v\in V$ with respect to the partition $\Pi$ is the vector
$r(v|\Pi)=(d(v,P_1),d(v,P_2),...,d(v,P_t))$, where $d(v,P_i)$ represents the distance between the vertex $v$ and
the set $P_i$. A partition  $\Pi$ of $V$ is a  \emph{resolving partition}  if different vertices of $G$ have different partition representations, i.e.,  for every pair
of vertices $u,v\in V$, $r(u|\Pi)\ne r(v|\Pi)$. The  \emph{partition
dimension} of $G$ is the minimum number of sets in any resolving
partition for $G$. In this paper we obtain several tight bounds
on the partition dimension of unicyclic graphs.\end{abstract}



\section{Introduction}

 The concepts of resolvability and location in graphs were described
independently by Harary and Melter \cite{harary} and Slater
\cite{leaves-trees}, to define the same structure in a
graph. After these papers were published several authors
developed diverse theoretical works about this topic, for instance,
\cite{pelayo1,pelayo2,chappell,chartrand,chartrand1,chartrand2,fehr,haynes,landmarks,YKRdim}.
 Slater described the usefulness of these ideas into long range
aids to navigation \cite{leaves-trees}.
Also, these concepts  have some applications in chemistry for representing chemical compounds \cite{pharmacy1,pharmacy2} or to problems of
pattern recognition and image processing, some of which involve the use of hierarchical data structures \cite{Tomescu1}.
Other applications of this concept to
navigation of robots in networks and other areas appear in
\cite{chartrand,robots,landmarks}. Some variations on resolvability
or location have been appearing in the literature, like those about
conditional resolvability \cite{survey}, locating domination
\cite{haynes}, resolving domination \cite{brigham} and resolving
partitions \cite{chappell,chartrand2,fehr,tomescu,partitionDimensionCorona,YRpd}.

Given a graph $G=(V,E)$ and a set of vertices
$S=\{v_1,v_2,...,v_k\}$ of $G$, the  \emph{metric representation} of a
vertex $v\in V$ with respect to $S$ is the vector
$r(v|S)=(d(v,v_1),d(v,v_2),...,d(v,v_k))$, where $d(v,v_i)$\footnote{To avoid ambiguity in some cases we will denote the distance between two vertices $u,v$ of a graph $G$ by $d_G(u,v)$.} denotes the distance between the vertices $v$ and
$v_i$, $1\le i\le k$. We say that $S$ is a  \emph{resolving set}  if different vertices of $G$ have different metric representations, i.e.,  for every
pair of vertices $u,v\in V$, $r(u|S)\ne r(v|S)$. The  \emph{metric
dimension}\footnote{Also called locating number.} of $G$ is the
minimum cardinality of any resolving set of $G$, and it is
denoted by $\dim(G)$. 
%The metric dimension of graphs is studied, for instance, in
%\cite{pelayo1,pelayo2,chappell,chartrand,chartrand1,PartitionDimTrees,tomescu}.

Given an ordered partition $\Pi =\{P_1,P_2, ...,P_t\}$ of the
vertices of $G$, the  \emph{partition representation} of a vertex
$v\in V$ with respect to the partition $\Pi$ is the vector
$r(v|\Pi)=(d(v,P_1),d(v,P_2),...,d(v,P_t))$, where $d(v,P_i)$, with
$1\leq i\leq t$, represents the distance between the vertex $v$ and
the set $P_i$, i.e.,  $d(v,P_i)=\min_{u\in P_i}\{d(v,u)\}$. We say
that $\Pi$ is a \emph{resolving partition}  if different vertices of $G$ have different partition representations, i.e.,  for every pair
of vertices $u,v\in V$, $r(u|\Pi)\ne r(v|\Pi)$. The  \emph{partition
dimension} of $G$ is the minimum number of sets in any resolving partition for $G$ and it is denoted by $\pd(G)$.

The partition
dimension of graphs was studied in \cite{chappell,chartrand2,fehr,Saenpholphat1,tomescu,partitionDimensionCorona,YRpd}.
For instance, Chappell, Gimbel and Hartman obtained several
relationships between metric dimension, partition dimension, diameter, and other graph parameters  \cite{chappell}.   Charttrand, Zhang and Salehi showed that for every nontrivial graph $G$ it follows  $\pd(G)\le \pd(G\square K_2)$ and they also showed that for an induced subgraph $H$ of a connected graph $G$ the ratio $r_p=\pd(H)/\pd(G)$ can be arbitrarily large  \cite{chartrand2}.
The partition dimension of some specific families of graphs was studied further in a number
of other papers. For instance, Cayley digraphs were studied by Fehr, Gosselin and Oellermann \cite{fehr}, the infinite graphs ($\mathbb{Z}^2$, $\xi_4$) and ($\mathbb{Z}^2$, $\xi_8$) (where the set of vertices is the set of points of the integer lattice and the set of edges consists
of all pairs of vertices whose city block and chessboard distances, respectively, are $1$) were studied by Tomescu \cite{tomescu}, corona product graphs were studied by Rodr\'{\i}guez-Vel\'{a}zquez,  Yero and  Kuziak \cite{partitionDimensionCorona} and Cartesian product graphs were studied by Yero and Rodr\'{\i}guez-Vel\'{a}zquez \cite{YRpd}.
Here we study  the partition dimension of unicyclic graphs.  A similar study on the metric dimension was previously done by Poisson and Zhang \cite{PartitionDimTrees}.

\section{Results}

The set of all spanning trees of a connected graph $G$ is denoted by ${\cal T}(G)$.
It was shown in \cite{chartrand} that if $G$ is a connected unicyclic graph  of order at least $3$ and $T\in {\cal T}(G)$,  then
\begin{equation}\label{dimensionUniTrees}
\dim(T) - 2\le \dim(G)\le \dim(T) + 1.
\end{equation}
A formula for the dimension of trees that are not  paths has been established in \cite{chartrand,harary,leaves-trees}. In order to present this formula, we need additional definitions. A vertex of degree at least $3$ in a graph $G$ will be called a \emph{major vertex} of $G$.
Any pendant vertex $u$ of  $G$ is said to be a \emph{terminal vertex} of a major vertex $v$ of $G$ if
$d(u, v)<d(u,w)$ for every other major vertex $w$ of $G$. The \emph{terminal degree}  of
a major vertex $v$ is the number of terminal vertices of $v$. A major vertex $v$ of $G$ is an
\emph{exterior major vertex} of $G$ if it has positive terminal degree.
%%%%%%%%%%%%%%%%%%%%%%%%%%%%%%%%%%%%%%%%%
\begin{figure}[h]
  \centering
  \includegraphics[width=0.3\textwidth]{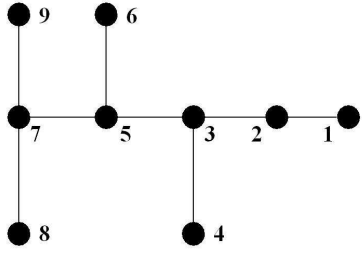}
  \caption{In this tree the vertex 3 is an exterior major vertex of terminal degree two: 1 and 4 are terminal vertices of 3. }\label{segundomejor}
\end{figure}

%%%%%%%%%%%%%%%%%%%%%%%%%%%%%%%%%%%%%%
Let $n_1(G)$ denote the number of pendant vertices of  $G$, and let $\ex(G)$ denote the number
of exterior major vertices of $G$. We can now state the formula for the dimension of a tree \cite{chartrand,harary,leaves-trees}:  if $T$ is a tree that is not a path, then
\begin{equation}\label{bounDimUsandoExterior}\dim(T) = n_1(T) - \ex(T).\end{equation}
Thus, by the above result and (\ref{dimensionUniTrees})  we have that if $G$ is a connected unicyclic graph  of order at least $3$ and $T\in {\cal T}(G)$, then
\begin{equation}\label{n1-ex-plus-one}
n_1(T) - \ex(T) -2 \le \dim(G)\le  n_1(T) - \ex(T) + 1.
\end{equation}

\noindent
\textbf{Example}. Let $G$ be a graph obtained in the following way: we begin with a cycle $C_4=u_1u_2u_3u_4u_1$ and, then we add vertices
$v_{1},\ldots, v_{k}, k\geq 2$, and edges $u_1v_{i}, 1\leq i\leq k$. Thus, $\dim(G)=k+1$. Now, let $T\in \mathcal{T}(G)$ obtained by deleting the edge $u_4u_1$ in the cycle. Hence, we have $n_1(T)=k+1$ and $\ex(T)=1$. So, the above upper bound is tight.

%The above bound is achieved for the graph in Figure \ref{1uniciclico} and it is not achieved for the graph in Figure \ref{otraetiqueta}.

%\begin{figure}[h]
 % \centering
 % \includegraphics[width=0.3\textwidth]{triangulo-4hojas}
  %\caption{$\Pi=\{\{1,4,8,12\},\{2,5,9,13\},\{3,6,10,14\},\{7,11,15\}\}$ is a resolving partition. }\label{otraetiqueta}
%\end{figure}

%%%%%%%%%%%%%%%%%%%%%%%%%%%%%%%%%%%%%%%%%%%%%%%%%%%

It is natural to think that the partition dimension and metric dimension are related; it was shown in  \cite{chartrand2} that for any
nontrivial connected graph $G$ we have
\begin{equation}\label{partition-dimension}
\pd(G)\le  \dim(G) + 1.
\end{equation}
As a consequence of (\ref{n1-ex-plus-one}), if $G$ is a connected unicyclic graph and $T\in {\cal T}(G)$, then
\begin{equation}\label{partition-dimension1}
\pd(G)\le   n_1(T) - \ex(T) + 2.
\end{equation}

The following well-known claim is very easy to verify.

\begin{claim} \label{distancescycles} Let $C$ be a cycle graph. If $x,y,u$ and $v$ are vertices of $C$ such that $x$ and $y$ are adjacent and $d(u,x)=d(v,x)$, then $d(u,y)\neq d(v,y)$.
\end{claim}

\begin{corollary} \label{dimCiclo} For any cycle graph $C$,   $\dim(C)=2$.
\end{corollary}

Any vertex adjacent to a pendant vertex of  a graph G is called a support vertex of $G$. Let $\rho(G)$ be the number of support vertices of $G$ adjacent to more than one pendant vertex.
% and let $\varepsilon (G)$ the minimum number of leaves in any spanning tree of $G$, i.e.,
% $$\varepsilon(G)=\displaystyle\min_{T\in {\cal T}(G)}\left\{n_1(T)\right\}.$$

\begin{theorem} Let  $G$ be a connected unicyclic graph. If every vertex belonging to the cycle of $G$ has degree greater than two,   then  $$\dim(G)\le n_1(G)-\rho(G).$$
\end{theorem}

\begin{proof}
%The result follows for the cycle graphs $G=C_n$, so we suppose $G\ne C_n$. Let $T\in {\cal T}(G)$ such that $\varepsilon(G)=n_1(T)$.
Let $C$ be the set of vertices belonging to the cycle of $G$. In order to show that the set of pendant vertices of $G$ is a resolving set, we only need to show that for every $u,v\in C$  we can find two pendant vertices, $x, y$,   such that if $d_G(u,x)=d_G(v,x)$, then $d_G(u,y)\neq d_G(v,y)$. To begin with, for every pendant vertex $w$  we define $w_c$ as the vertex of $C$ such that $d_G(w,w_c)=d_G(w,C)$.

%Case 1: There are two adjacent vertices of degree greater than two belonging to $C$. In this case

We  take $x,y$ as two pendant vertices of $G$ such that $x_c$ and $y_c$ are adjacent vertices.
Note that in this case  for every $u,v\in C$ we have $d_G(u,x)=d_G(u,x_c)+d_G(x_c,x)$,  $d_G(u,y)=d_G(u,y_c)+d_G(y_c,y)$, $d_G(v,x)=d_G(v,x_c)+d_G(x_c,x)$ and $d_G(v,y)=d_G(v,y_c)+d_G(y_c,y)$. So, if $d_G(u,x)=d_G(v,x)$, we conclude $d_G(u,y)\neq d_G(v,y)$.
Thus, the set of pendant vertices of $G$ is a resolving set.

%Case 2: There are no adjacent vertices of degree greater than two belonging to $C$. We take $y$ as a vertex of degree one in $G$ and
%let $x$ be a vertex of degree two belonging to $C$ such that  $y_c$ is  adjacent to $x$  ($y_c$ has degree  greater than two). So we take $T=G-\{xy_c\}$. Now we have  $d_G(u,y)=d_G(u,y_c)+d_G(y_c,y)$ and $d_G(v,y)=d_G(v,y_c)+d_G(y_c,y)$. As above,  if $d_G(u,x)=d_G(v,x)$, we conclude $d_G(u,y)\neq d_G(v,y)$.
%Thus, the set of leaves of $T$ is a resolving set for $G$.

If we consider pendant vertices as being equivalent if they have the same support vertex, then a resolving set of minimum cardinality should contain all but one of these pendant vertices per equivalent class. Thus, the result follows.
\end{proof}

The above bound is tight, it is achieved for the graph in Figure \ref{1uniciclico}.

\begin{figure}[h]
 \centering
  \includegraphics[width=0.3\textwidth]{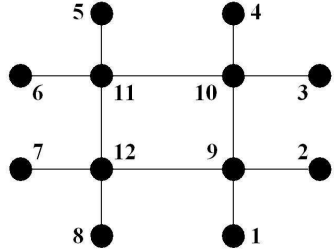}\\
  \caption{$\Pi=\{\{3,5,7,9\},\{1,6,8,10\},\{2,4,11\},\{12\}\}$ is a resolving partition and
$\{1,3,5,7\}$ is a
resolving set.}\label{1uniciclico}
\end{figure}

\begin{corollary}\label{coro-pd-varepsilon-rho}
Let  $G$ be a connected unicyclic graph. If every vertex belonging to the cycle of $G$ has degree greater than two,   then $$\pd(G)\le n_1(G)-\rho(G)+1.$$
\end{corollary}

Note that for the graph in Figure \ref{1uniciclico}, Corollary \ref{coro-pd-varepsilon-rho} leads to $\pd(G)\le 5$, while bound (\ref{partition-dimension1}) only gives $\pd(G)\le 6$.

In order to obtain other results we need to introduce some additional notations.
Let $S=\{s_1,s_2,...,s_{\kappa(G)}\}$ be the set of exterior major
vertices of the unicyclic graph $G$ with terminal degree greater than one.  For every $s_i\in S$, let $\{s_{i1},s_{i2},..., s_{il_i}\}$ be the set
of terminal vertices of $s_i$ and let $\tau(G)=\displaystyle\max_{i\in \{1,...,\kappa(G)\}}\{l_i\}$.

\begin{lemma}{\rm \cite{chartrand2}}  \label{pd(pn)=2}
 Let G be a connected graph of order $n\ge  2$. Then $pd(G) = 2$ if
and only if $G\cong P_n$.
\end{lemma}

\begin{theorem}\label{ThPdcycles}
Let $G$ be a connected unicyclic graph.
\begin{itemize}
\item[{\rm (i)}]If $G$ is a cycle graph or every exterior major vertex of $G$ has terminal degree  one, then $$\pd(G)=3.$$
\item[{\rm (ii)}]If $G$ contains at least an exterior major vertex of terminal degree greater than one, then  $$\pd(G)\le
\kappa(G)+\tau(G)+1.$$
\end{itemize}
\end{theorem}

\begin{proof}
Let us prove (i). If $G$ is a cycle graph, then  by (\ref{dimensionUniTrees}), Corollary \ref{dimCiclo} and Lemma \ref{pd(pn)=2}  we obtain $\pd(G)=3.$
 Now we consider that every exterior major vertex of $G=(V,E)$ has terminal degree  one. Notice that every exterior  major vertex  $u$  has degree three and it belongs to the cycle $C$ of $G$.
Let $\{c_0,c_1,...,c_{k-1}\}$ be the set of vertices of $G$ belonging to $C$ where $c_i$ and $c_{i+1}$ are adjacent (the subscripts are taken module  $k$). Without loss of generality we can suppose $c_0$ has terminal degree one. For every  exterior major vertex $c_i$, $W_i$ will denote the set of vertices belonging to the path starting at $c_i$ and ending at its terminal vertex. For every $c_j$ of degree two we assume $W_j=\{c_j\}$.

For  $k$ even we claim that $\Pi=\{W_0,A_2,A_3\}$ is a resolving partition for $G$, where $A_2=W_{\frac{k}{2}}\cup W_{\frac{k}{2} +1}$ and $A_3=V-\left(W_0\cup W_{\frac{k}{2}}\cup W_{ \frac{k}{2} +1}\right)$. To show this we differentiate three cases for $x,y\in V$.
\\[1ex]
Case 1: $x,y\in W_0$. Since  $d(x,c_0)\ne d(y,c_0)$, we conclude $d(x,A_3)=d(x,c_0)+1\ne d(y,c_0)+1= d(y,A_3)$.
\\[1ex]
Case 2: $x,y\in A_2$. If $d\left(x,c_{_{ \frac{k}{2} -1}}\right)=d\left(y,c_{ _{\frac{k}{2} -1}}\right)$, then either
$$d\left(x,c_{_{ \frac{k}{2} +2}}\right)=d\left(x,c_{_{ \frac{k}{2} -1}}\right)-1  = d\left(y,c_{_{ \frac{k}{2} -1}}\right)-1= d\left(y,c_{_{ \frac{k}{2} +2}}\right)-2$$
 or $$d\left(x,c_{_{ \frac{k}{2} +2}}\right)=d\left(x,c_{_{ \frac{k}{2} -1}}\right)+1  = d\left(y,c_{_{ \frac{k}{2} -1}}\right)+1= d\left(y,c_{_{ \frac{k}{2} +2}}\right)+2.$$ Thus, since $c_{_{ \frac{k}{2} +2}}\in A_3$ for $k\ge 6$ and  $c_{_{ \frac{k}{2} +2}}\in W_0$ for $k=4$, we have $d\left(x,A_3\right)\ne d\left(y,A_3\right)$ or $d\left(x,W_0\right)\ne d\left(y,W_0\right)$. On the other hand, since $c_{_{\frac{k}{2} -1}}\in A_3$, if $d\left(x,c_{_{ \frac{k}{2} -1}}\right)\ne d\left(y,c_{_{ \frac{k}{2} -1}}\right)$ and $d\left(x,c_{_{ \frac{k}{2} +1}}\right)= d\left(y,c_{_{ \frac{k}{2} +1}}\right)$, then $d\left(x,A_3\right)\ne d\left(y,A_3\right)$. In the case  $d\left(x,c_{_{ \frac{k}{2} +1}}\right)\ne d\left(y,c_{_{ \frac{k}{2} +1}}\right)$ we have $d\left(x,W_0\right)=d\left(x,c_{_{ \frac{k}{2} +1}}\right)+d\left(c_{_{ \frac{k}{2} +1}},c_0\right)\ne d\left(y,c_{_{ \frac{k}{2} +1}}\right)+d\left(c_{_{ \frac{k}{2} +1}},c_0\right)= d\left(y,W_0\right)$.
\\[1ex]
Case 3: $x,y\in  A_3$. Let $x\in W_i$ and $y\in W_j$. If $i=j$, then $d(x,W_0)\ne d(y,W_0)$ and $d(x,A_2)\ne d(y,A_2)$. Now we consider the next cases.
\\[1ex]
Case 3.1: $0<i<j<k/2$. If $d(y,A_2)=d(x,A_2)$, then we have $d(y,c_j)+d(c_j,c_{k/2})=d(y,c_{k/2})=d(x,c_{k/2})=d(x,c_{i})+d(c_i,c_{j})+d(c_j,c_{k/2})$. So, $d(y,c_j)=d(x,c_{i})+d(c_i,c_{j})$ and we obtain the following.
$$d(x,c_0)=d(x,c_i)+d(c_i,c_0)=d(y,c_j)-d(c_i,c_{j})+d(c_i,c_0)\ne d(y,c_j)+d(c_j,c_i)+d(c_i,c_0)=d(y,c_0).$$
Thus, $d(x,W_0)\ne d(y,W_0)$.
\\[1ex]
Case 3.2: $\frac{k}{2}+1<i<j\le k-1$. Proceeding analogously to Case 3.1, if $d(y,A_2)=d(x,A_2)$, then we obtain that $d(x,W_0)\ne d(y,W_0)$.
\\[1ex]
Case 3.3: $0<i<k/2$ and $\frac{k}{2}+1<j\le k-1$. If $d(x,A_2)=d(y,A_2)$, then we have $d(x,c_{i})+d(c_i,c_{k/2})=d(x,c_{k/2})=d(y,c_{k/2+1})=d(y,c_j)+d(c_j,c_{\frac{k}{2}+1})$. Thus,
\begin{align*}
d(x,c_0)&=d(x,c_i)+d(c_i,c_0)\\
&=d(y,c_j)+d(c_j,c_{\frac{k}{2}+1})-d(c_i,c_{k/2})+d(c_i,c_0)\\
&=d(y,c_j)+d(c_0,c_{\frac{k}{2}+1})-d(c_0,c_j)-d(c_i,c_{k/2})+d(c_i,c_0)\\
&=d(y,c_j)+d(c_0,c_j)+d(c_0,c_{\frac{k}{2}+1})-2d(c_0,c_j)-d(c_i,c_{k/2})+d(c_i,c_0)\\
&=d(y,c_0)+\left(\frac{k}{2}-1\right)-2(k-j)-\left(\frac{k}{2}-i\right)+i\\
&=d(y,c_0)+2(i+j)-2k-1.
\end{align*}
Hence, if $i+j\le k$, then $2(i+j)-2k-1<0$ and, as a consequence,  $d(x,c_0)<d(y,c_0)$. Analogously, if $i+j\ge k+1$, then $2(i+j)-2k-1>0$, so we have $d(x,c_0)>d(y,c_0)$. As a result,  $d(x,W_0)\ne d(y,W_0)$.

On the other hand, suppose $k$ is odd. If $k=3$, then it is straightforward to check that $\{W_0,W_1,W_2\}$ is a resolving partition for $G$. So we assume $k\ge 5$ and we claim that $\Pi=\{B_1,B_2,B_3\}$ is a resolving partition for $G$, where $B_1=W_0\cup W_1$, $B_2=W_{\lfloor k/2 \rfloor}\cup W_{\lceil k/2 \rceil}$ and $B_3=V-(B_1\cup B_2)$. To show this we consider two different vertices $x,y\in V$ and as above we take $x\in W_i$ and $y\in W_j$. If $i=j$, then $x,y\in B_l$ for some $l\in\{1,2,3\}$ and $d(x,B_r)\ne d(y,B_r)$ for any $r\in \{1,2,3\}-\{l\}$. Now on we assume $i<j$ and we differentiate the following three cases.
\\[1ex]
Case 1': $x,y\in B_1$. Since $i<j$ and $B_1=W_0\cup W_1$ we have $i=0$ and $j=1$. If $k=5$, then $d(x,B_3)=d(y,B_3)$ implies $d(x,B_2)=d(y,B_2)+2$. So we consider $k\ge 7$. Now $d(x,B_3)=d(y,B_3)$ implies $d(x,c_0)=d(y,c_1)$. Thus,
\begin{align*}
d(x,c_{\lceil k/2 \rceil})&=d(x,c_0)+d(c_0,c_{\lceil k/2 \rceil})\\
&=d(x,c_0)+d(c_0,c_{\lfloor k/2 \rfloor})\\
&=d(x,c_0)+d(c_1,c_{\lfloor k/2 \rfloor})+1\\
&=d(y,c_1)+d(c_1,c_{\lfloor k/2 \rfloor})+1\\
&=d(y,c_{\lfloor k/2 \rfloor})+1\\
&>d(y,c_{\lfloor k/2 \rfloor}).
\end{align*}
Hence, we obtain that $d(x,B_2)\ne d(y,B_2)$.
\\[1ex]
Case 2': $x,y\in B_2$. Proceeding analogously to Case 1' we obtain that if $d(x,B_3)=d(y,B_3)$, then $d(x,B_2)\ne d(y,B_2)$.
\\[1ex]
Case 3': $x,y\in  B_3$. Now we consider the following cases.
\\[1ex]
Case 3'.1: $1<i<j<\lfloor k/2 \rfloor$. If $d(y,B_2)=d(x,B_2)$, then we have $d(y,c_j)+d(c_j,c_{\lfloor k/2 \rfloor})=d(y,c_{\lfloor k/2 \rfloor})=d(x,c_{\lfloor k/2 \rfloor})=d(x,c_{i})+d(c_i,c_{j})+d(c_j,c_{\lfloor k/2 \rfloor})$. So, $d(y,c_j)=d(x,c_{i})+d(c_i,c_{j})$ and we obtain
\begin{align*} d(y,c_1)&=d(y,c_j)+d(c_j,c_i)+d(c_i,c_1)\\
&=d(x,c_i)+2d(c_j,c_i)+d(c_i,c_1)\\
&=d(x,c_1)+2d(c_j,c_i).
\end{align*}
Thus, $d(x,B_1)\ne d(y,B_1)$.
\\[1ex]
Case 3'.2: $\lceil k/2 \rceil<i<j\le k-1$. Proceeding as in Case 3'.1 we have that if $d(y,B_2)=d(x,B_2)$, then we obtain that $d(x,B_1)\ne d(y,B_1)$.
\\[1ex]
Case 3'.3: $1<i<\lfloor k/2 \rfloor$ and $\lceil k/2 \rceil<j\le k-1$.
If $d(x,B_2)=d(y,B_2)$, then we have  $d(x,c_{i})+d(c_i,c_{\lfloor k/2 \rfloor})=d(x,c_{\lfloor k/2 \rfloor})=d(y,c_{\lceil k/2 \rceil})=d(y,c_j)+d(c_j,c_{\lceil k/2 \rceil})$. Thus,
\begin{align*}
d(x,c_1)&=d(x,c_i)+d(c_i,c_1)\\
&=d(y,c_j)+d(c_j,c_{\lceil k/2 \rceil})-d(c_i,c_{\lfloor k/2 \rfloor})+d(c_i,c_1)\\
&=d(y,c_j)+d(c_0,c_{\lceil k/2 \rceil})-d(c_0,c_j)-d(c_i,c_{\lfloor k/2 \rfloor})+d(c_i,c_1)\\
&=d(y,c_j)+d(c_0,c_j)+d(c_0,c_{\lceil k/2 \rceil})-2d(c_0,c_j)-d(c_i,c_{\lfloor k/2 \rfloor})+d(c_i,c_1)\\
&=d(y,c_0)+\lfloor k/2 \rfloor-2(k-j)-\left(\lfloor k/2 \rfloor-i\right)+(i-1)\\
&=d(y,c_0)+2(i+j-k)-1.
\end{align*}
Hence, if $i+j\le k$, then $2(i+j-k)-1<0$ and, as a consequence,  $d(x,c_1)<d(y,c_0)$. Analogously, if $i+j\ge k+1$, then $2(i+j-k)-1>0$, so we have $d(x,c_1)>d(y,c_0)$. As a result, $d(x,B_1)\ne d(y,B_1)$.

Therefore, for every $x,y\in V$, $x\ne y$, we have $r(x|\Pi)\ne r(y|\Pi)$ and, as a consequence, $\pd(G)\le 3$. By Lemma \ref{pd(pn)=2} we know that for every graph $G$ different from a path we have $\pd(G)\ge 3$,  hence we obtain  $\pd(G)=3$.

Now, let us prove  (ii).
Let $s_l\in S$ be an arbitrary exterior major vertex of $G=(V,E)$, with terminal degree greater than one. Let $u\in V$ be the vertex of the cycle $C$ in $G$, such that $d(u,s_l)=\min_{v\in C}\{d(v,s_l)\}$. Let $v\in C$ such that $u$ is adjacent to $v$. For a terminal vertex
$s_{ij}$ of an exterior major vertex $s_i$ we denote by $S_{ij}$ the set of vertices of $G$, different from $s_i$, belonging to the $s_i-s_{ij}$ path. If $l_i< \tau(G)$, we assume $S_{ij}=\emptyset$ for every $j\in \{l_i+1,...,\tau(G)\}$.
Now, let $A=\{v\}$ and $B=C-\{v\}$. Let $A_i=S_{i1}$, for every $i\in \{1,...,\kappa(G)\}$ and if $\tau(G)\ge 3$, then let $B_j=\bigcup_{i=1}^{\kappa(G)}S_{ij}$, for every $j\in \{2,...,\tau(G)-1\}$. Now we will show that the partition $\Pi=\{A, B, A_1,A_2,....,A_{\kappa(G)},B_2,B_3,...,B_{\tau(G)-1},R\}$, with $R=V(G)-A-B-\bigcup_{i=1}^{\kappa(G)}A_i-\bigcup_{i=2}^{\tau(G)-1}B_i$, is a resolving partition for $G$. Notice that the sets $B_j$ could not exist in the case $\tau(G)=2$. Hence, $R$ collects all major vertices of terminal degree one and the attached terminals.
 Let $x,y\in V$ be two different vertices in $G$. We have the following cases.
\\[1ex]
Case 1: If $x,y\in A_i$, then $d(x, R)\ne d(y,R)$. Namely, any path from $x$ or $y$ to $R$ must contain $s_i$.
\\[1ex]
Case 2: $x,y\in B$. If $d(x,v)\ne d(y,v)$, then $d(x,A)\ne d(y,A)$. On the contrary, if $d(x,v)=d(y,v)$, then $d(x,u)\ne d(y,u)$ due to Claim~\ref{distancescycles}. So, for $s_l\in S$ we have $A_l=S_{l1}$ and we obtain that
$$d(x,A_l)=d(x,u)+d(u,S_{l1})\ne d(y,u)+d(u,S_{l1})=d(y,A_l).$$
% \\[1ex]
\noindent
Case 3: $x,y\in B_j$. If $x,y\in S_{ij}$, then $x$ belongs to the $y-s_i$ path or $y$ belongs to the $x-s_i$ path. In both cases we have $d(x,A_i)=d(x,s_i)+1\ne d(y,s_i)+1=d(y,A_i)$. On the contrary, if $x\in S_{ij}$ and $y\in S_{kj}$, $i\ne k$, then let us suppose $d(x,A_i)=d(y,A_i)$. So, we have
\begin{align*}
d(x,A_k)&=d(x,s_i)+d(s_i,s_k)+1\\
&=d(x,A_i)+d(s_i,s_k)\\
&=d(y,A_i)+d(s_i,s_k)\\
&=d(y,s_k)+2d(s_i,s_k)+1\\
&=d(y,A_k)+2d(s_i,s_k)\\
&>d(y,A_k).
\end{align*}
\noindent
Case 4: $x,y\in R$. Let $a,b\in C$ such that $d(x,a)=\min_{c\in C}\{d(x,c)\}$ and $d(y,b)=\min_{c\in C}\{d(y,c)\}$. If $d(x,a)\ne d(y,b)$ and $a,b\ne v$, then $d(x,B)\ne d(y,B)$. Also, if $d(x,a)\ne d(y,b)$ and ($a=v$ or $b=v$), then we have either $d(x,A)\ne d(y,A)$ or $d(x,B)\ne d(y,B)$. Now, let us suppose $d(x,a)=d(y,b)$. We have the following cases.
\\[1ex]
Subcase 4.1: $a=b$. Hence, we consider a terminal vertex $s_{i1}$, such that $d(x,s_{i1})+d(y,s_{i1})=\min_{l\in \{1,...,\kappa(G)\}}\{d(x,s_{l1})+d(y,s_{l1})\}$. Let the vertices $c,d$ belonging to the $a-s_{i1}$ path $P$, with $d(x,c)=\min_{w\in P}\{d(x,w)\}$ and $d(y,d)=\min_{w\in P}\{d(y,w)\}$. If $c=d$, then there exists a terminal vertex $s_{j1}$ such that either $x$ belongs to the $y-s_{j1}$ path or $y$ belongs to the $x-s_{j1}$ path and we have either $d(x,A_j)<d(y,A_j)$  or $d(y,A_j)<d(x,A_j)$. If there exists not such a terminal vertex $s_{j1}$, then we have that $x\in S_{i\tau(G)}$ and $y\in S_{j\tau(G)}$ for some $i\ne j$. Thus we have the following.
\begin{align*}
d(x,A_i)&=d(x,s_i)+1\\
&=d(x,a)-d(s_i,a)+1\\
&=d(y,a)-d(s_i,a)+1\\
&=d(y,a)+d(a,s_i)-2d(s_i,a)+1\\
&=d(y,A_i)-2d(s_i,a)\\
&<d(y,A_i).
\end{align*}
\noindent
On the other hand, if $c\ne d$, then we have either, ($d(x,a)=d(x,c)+d(c,d)+d(d,a)$ and $d(y,a)=d(y,d)+d(d,a)$) or ($d(y,a)=d(y,d)+d(d,c)+d(c,a)$ and $d(x,a)=d(x,c)+d(c,a)$). Let us suppose, without loss of generality, $d(x,a)=d(x,c)+d(c,d)+d(d,a)$ and $d(y,a)=d(y,d)+d(d,a)$.  Thus, we have
\begin{align*}
d(x,A_i)&=d(x,c)+d(c,A_i)\\
&=d(x,a)-d(c,d)-d(d,a)+d(c,A_i)\\
&=d(y,a)-d(c,d)-d(d,a)+d(c,A_i)\\
&=d(y,d)+d(d,a)-d(c,d)-d(d,a)+d(c,A_i)\\
&=d(y,d)-d(c,d)+d(c,A_i)\\
&=d(y,d)+d(d,c)+d(c,A_i)-2d(c,d)\\
&=d(y,A_i)-2d(c,d)\\
&<d(y,A_i).
\end{align*}
\noindent
Subcase 4.2: $a\ne b$. If $a=u$ or $b=u$, then let us suppose, for instance $b=u$. Let $Q$ be a shortest path between $a$ and $s_{l1}$. Let $c$ belonging to $Q$, such that $d(y,c)$ is the minimum value between the distances from $y$ to any vertex of $Q$. So, we have
\begin{align*}
d(x,A_l)&=d(x,a)+d(a,u)+d(u,c)+d(c,A_l)\\
&=d(y,b)+d(a,u)+d(u,c)+d(c,A_l)\\
&=d(y,c)+d(c,u)+d(a,u)+d(u,c)+d(c,A_l)\\
&=d(y,A_l)+2d(c,u)+d(a,u)\\
&>d(y,A_l).
\end{align*}
Now, let us suppose $a\ne u$ and $b\ne u$. If $d(a,v)\ne d(b,v)$, then $d(x,A)\ne d(y,A)$. On the contrary, if $d(a,v)=d(b,v)$, then $d(a,u)\ne d(b,u)$ (Claim~\ref{distancescycles}). So, we have
\begin{align*}
d(x,A_l)&=d(x,a)+d(a,u)+d(u,A_l)\\
&=d(y,b)+d(a,u)+d(u,A_l)\\
&\ne d(y,b)+d(b,u)+d(u,A_l)\\
&=d(y,A_l).
\end{align*}

Therefore, for every different vertices $x,y\in V$ we have $r(x|\Pi)\ne r(y|\Pi)$ and $\Pi$ is a resolving partition for $G$ and, as a consequence, (ii) follows.
\end{proof}

\noindent
\textbf{Example}. Let $G$ be a graph obtained in the following way: we begin with a cycle $C_k$, $k\ge 4$, and, then for each vertex $v$ of the cycle we add $k$ vertices $v_1,v_2,...,v_k$ and edges $vv_{i}, 1\leq i\leq k$. Thus, $G$ has $k^2$ vertices of degree one and $k$ exterior major vertices of terminal degree $k$.
Notice that the above bound leads to $\pd(G)\le 2k+1$ while (\ref{partition-dimension1}) gives $\pd(G)\le k^2-k+2$ and Corollary \ref{coro-pd-varepsilon-rho} gives $\pd(G)\le k^2-k+1$.

%%%%%%%%%%%%%%%%%%%%%%%%%%%%%%%%%%%%%%%%%%%%%%%%%%%%%%%%%%%%%%%%%%%%%%%%%%
%%%%%%%%%%%%%%%%%%%%%%%%%%%%%%%%%%%%%%%%%%

For a connected unicyclic graph $G$, let $\varepsilon (G)$ the minimum number of leaves in any spanning tree of $G$, i.e.,
$$\varepsilon(G)=\displaystyle\min_{T\in {\cal T}(G)}\left\{n_1(T)\right\}.$$
Now, for a spanning tree $T\in \mathcal{T}(G)$, let $\kappa(T)$ be the number of exterior major vertices of $T$, with terminal degree greater than one and let $\tau(T)$ be the maximum terminal degree of any exterior major vertex of $T$. Note that $\kappa(G)\le \kappa(T)$ and $\tau(G)\le \tau(T)$.

\begin{corollary} \label{corollaryKappaTauUnicyclic}
Let $G$ be a connected  unicyclic graph. For every  $T\in \mathcal{T}(G)$ such that $\varepsilon(G)=n_1(T)$,  $$\pd(G)\le
\kappa(T)+\tau(T)+1.$$
\end{corollary}

%\begin{proof}
%The result follows for the cycle graphs $G=C_n$, so we suppose $G\ne C_n$. If  every  exterior major vertex of $G$ has terminal degree one, then $1\le \kappa(T)\le 2$, $2\le \tau(T)\le 3$  and $\pd(G)=3$. So the result follows. Now, if $G$  contains at least an exterior major vertex of terminal degree greater than one, then $\kappa(G)\le \kappa(T)$ and $\tau(G)\le \tau(T)$. Thus,  we have $\pd(G)\le
%\kappa(G)+\tau(G)+1\le \kappa(T)+\tau(T)+1.$
%\end{proof}

For the unicyclic graph $G$ and a spanning tree $T\in \mathcal{T}(G)$, let $\xi(T)$ be the number of support vertices of $T$ and  $\theta(T)$ be  the maximum number of leaves adjacent to any support vertex of $T$. As a consequence of the above corollary we obtain the following result.

\begin{remark}
If $T$ is a spanning tree of a unicyclic graph $G$ such that $\varepsilon(G)=n_1(T)$, then $$\pd(G)\le \xi(T)+\theta(T)+1.$$
\end{remark}

\begin{proof}
If $T$ is a path, then $\xi(T)=2$ and $\theta(T)=1$, so the result follows. Now we suppose $T$ is not a path.
Let $v$ be an exterior major vertex of terminal degree $\tau(T)$ in $T$. Let $x$ be the number of leaves of $T$ adjacent to $v$ and let $y=\tau(T)-x$. Since $\kappa(T)+y\le \xi(T)$ and $x\le \theta(T)$, we deduce $\kappa(T)+\tau(T)\le \xi(T)+\theta(T)$. Thus the result follows from Corollary \ref{corollaryKappaTauUnicyclic}.
\end{proof}

As the next theorem shows, the above result can be improved.

\begin{theorem}\label{ThThetaXiUnicyclic}
If $T$ is a spanning tree of a unicyclic graph $G$ such that $\varepsilon(G)=n_1(T)$, then $$\theta(T)-1\le \pd(G)\le \xi(T)+\theta(T).$$
\end{theorem}

\begin{proof}
The result follows for the cycle graphs $G=C_n$, so we suppose $G\ne C_n$. Notice that different leaves adjacent to the same support vertex must belong to different sets of a resolving partition. Also, as $\varepsilon(G)=n_1(T)$ we have $\pd(G)\ge \theta(T)-1$. Thus, the lower bound follows. To obtain the upper bound, let $T\in {\cal T}(G)$ be such that $n_1(T)=\varepsilon(G)$. Let $C$ be the set of vertices belonging to the cycle of $G=(V,E)$ and let $uv\in E$, such that $u,v\in C$ and $T=G-\{uv\}$. Since $n_1(T)=\varepsilon(G)$, we have $\delta_G(v)\ge 3$ or $\delta_G(u)\ge 3$, where $\delta_G(u)$ represents the degree of the vertex $u$ in $G$. Now, let $S=\{s_1,s_2,...,s_{\xi(T)}\}$ be the set of support vertices of $T$, and for every $s_i\in S$, let $\{s_{i1},s_{i2},...,s_{il_i}\}$ be the set of  leaves of $s_i$ and let $\theta(T)=\max_{i\in \{1,...,\xi(T)\}}\{l_i\}$.

Let now $A_i=\{s_{i1}\}$, for every $i\in \{1,...,\xi(T)\}$. Let $M_{ij}=\{s_{ij}\}$, for every $j\in \{2,...,l_i\}$. If $l_i<\xi(T)$, then we assume $M_{ij}=\emptyset$, for every $j\in \{l_{i+1},...,\theta(T)\}$. Let $B_j=\bigcup_{i=1}^{\xi(T)}M_{ij}$, for every $j\in \{2,...,\theta(T)\}$. We will show that the partition $\Pi=\{A,A_1,A_2,....,A_{\xi(T)},B_2,B_3,...,B_{\theta(T)}\}$, with $A=V(G)-\bigcup_{i=1}^{\xi(T)}A_i-\bigcup_{i=2}^{\theta(T)}B_i$, is a resolving partition for $G$. Let $x,y\in V$ be two different vertices in $G$. We have the following cases.
\\[1ex]
Case 1: $x\not\in C$ and $y\in C$. If $\delta_G(u)=2$ and $\delta_G(v)\ge 3$, then $u$ is a leaf in $T$ and we can suppose, without loss of generality, that $u=s_{i1}$, for some $i\in \{1,...,\xi(T)\}$, so $A_i=\{u\}$. Hence, if $y=u$ or $x$ is a leaf, then $x$ and $y$ belong to different sets of $\Pi$. On the contrary, if $y\ne u$ and $x$ is not a leaf, then there exist a leaf $s_{l1}$ such that $x$ belongs to a minimum $y-s_{l1}$ path, thus $d_G(y,A_l)>d_G(x,A_l)$.
Now, if $\delta_G(u)\ge 3$ and $\delta_G(v)\ge 3$, then let $a\in C$ such that $d_G(x,a)=\min_{b\in C}\{d_G(x,b)\}$. Hence, there exists a leaf $s_{j1}$ such that $x$ belongs to the $a-s_{j1}$ path. So, we have
$$d_G(y,A_j)=d_G(y,a)+d_G(a,A_j)>d_G(y,a)+d_G(x,A_j)\ge d_G(x,A_j).$$
\\[1ex]
Case 2: $x\not\in C$ and $y\not\in C$. If $x,y\in B_j$, for some $j\in \{2,...,\theta(T)\}$, then $x=s_{ij}$ and $y=s_{kj}$, with $1\ne j\ne k\ne 1$. So, we have
\begin{align*}
d_G(y,A_i)&=d_G(y,s_k)+d_G(s_k,s_i)+1\\
&\ge d_G(y,s_k)+2\\
&=d_G(y,s_k)+d_G(x,A_i)\\
&>d_G(x,A_i).
\end{align*}
On the other side, if $x,y\in A$, then there exists a leaf $s_{i1}$ such that either, $x$
belongs to one $y-s_{i1}$ path or $y$ belongs to one $x-s_{i1}$ path. So, we have $d_G(x,A_i)\ne d_G(y,A_i)$.
\\[1ex]
Case 3: $x,y\in C$. Now we have the following subcases.
\\[1ex]
Subcase 3.1: $\delta_G(u)\ge 3$ and $\delta_G(v)\ge 3$. Let $s_{k1}$ and $s_{j1}$, $j\ne k$ be two leaves, such that the $v-s_{k1}$ path share with cycle $C$ only the vertex $v$ and the $u-s_{j1}$ path share with cycle $C$ only the vertex $u$. If $d_G(x,u)\ne d_G(y,u)$, then we have
$$d_G(x,A_j)=d_G(x,u)+d_G(u,s_{j1})\ne d_G(y,u)+d_G(u,s_{j1})=d_G(y,A_j).$$
On the contrary, if $d_G(x,u)=d_G(y,u)$, then $d_G(x,v)\ne d_G(y,v)$ and we have
$$d_G(x,A_k)=d_G(x,v)+d_G(v,s_{k1})\ne d_G(y,v)+d_G(v,s_{k1})=d_G(y,A_k).$$
\\[1ex]
Subcase 3.2: Without loss of generality, assume $\delta_G(u)=2$ and $\delta_G(v)\ge 3$. Hence, $u$ is a leaf in $T$ and we can suppose, without loss of generality, that $u=s_{i1}$, for some $i\in \{1,...,\xi(T)\}$, so $A_i=\{u\}$. If $x=u$ or $y=u$, then $x,y$ belong to different sets of $\Pi$. If $d_G(x,u)\ne d_G(y,u)$, then $d_G(x,A_i)\ne d_G(y,A_i)$. On the other hand, if $d_G(x,u)=d_G(y,u)$, then  $d_G(x,v)\ne d_G(y,v)$. Now, let $s_{k1}$ be a leaf, such that the $v-s_{k1}$ path share with cycle $C$ only the vertex $v$. Hence, we have
$$d_G(x,A_k)=d_G(x,v)+d_G(v,s_{k1})\ne d_G(y,v)+d_G(v,s_{k1})=d_G(y,A_k).$$
Therefore, for every different vertices $x,y\in V$ we have $r(x|\Pi)\ne r(y|\Pi)$ and $\Pi$ is a resolving partition for $G$.
\end{proof}

Note that the above bound is achieved for unicyclic graphs having at most two exterior major vertices and each one of them  has terminal degree  one. In such a case, $\pd(G)=3.$
%Notice that the above bound is attained for the graphs of Figure \ref{triangulos}.

%\begin{figure}
%  \centering
%  \includegraphics[width=0.3\textwidth]{triangulos}\\
%  \caption{$\{\{1\},\{3\},\{2,4\}\}$ and $\{\{3\},\{1,2\},\{4,5\}\}$ are resolving partitions of the first and second graphs, respectively.}\label{triangulos}
%\end{figure}

%%%%%%%%%%%%%%%%%%%%%%%%%%%%%%%%%%
%% Conjecture...

The following conjecture, if true, would be completely analogous to the estimate
known for the metric dimension.

\begin{conjecture}
If $T$ is a spanning tree of a unicyclic graph $G$,  then $$\pd(G)\le \pd(T)+1.$$
\end{conjecture}

According to (\ref{dimensionUniTrees}), Lemma \ref{pd(pn)=2} and Theorem  \ref{ThPdcycles} (i),   the above conjecture is true for every cycle graph and for every unicyclic graph where every exterior major vertex  has terminal degree  one.
Even so, the previous conjecture seems to be very hard to prove.
We therefore present the following weakened version.

\begin{proposition}
If $T$ is a spanning tree of a unicyclic graph $G$,  then $$\pd(G)\le
\pd(T)+3.$$
\end{proposition}

\begin{proof}
 Arbitrarily cut the cycle $C=\{c_0,\dots,c_{k-1}\}$ by deleting,
without loss of generality, $c_0c_{1}$. This results in a (spanning) tree $T$. Let
$\Pi$ be an optimum resolving partition for $T$, i.e.,
$\Pi=\{A_1,\dots,A_{\pd(T)}\}$. Let $D=\left\{c_0,c_1,c_{\left\lfloor\frac{k}
{2}\right\rfloor}\right\}$
and define $A_i^G=A_i-D$.
We claim that
$$\Pi^G=\left\{A_1^G,\dots,A^G_{\pd(T)},\{c_0\}, \{c_1\},\left\{c_{\left\lfloor\frac{k}{2}\right\rfloor}\right\}\right\}$$
is a resolving partition for $G$, where we only take the nonempty sets $A_i^G=A_i-D$. So, the order  of this partition may be less than  $\pd(T)+3$.

For every $i\in \{0,1,...,k-1\} $, let $T_i=(V_i,E_i)$ be the subtree of $G$ rooted at $c_i$. Note that may occur that $V_i=\{c_i\}$.
We differentiate two cases for  $x,y\in V(G)$, $x\ne y$.

Case 1. $x,y\in V_i$. If $d_G(x,c_i)\ne d_G(y,c_i)$, then $d_G(x,c_0)\ne d_G(y,c_0)$. Now, if $d_G(x,c_i)=d_G(y,c_i)$, then for every $v\in V(G)-V_i$, it follows $d_G(x,v)=d_G(y,v)$ (notice that $d_T(x,v)=d_T(y,v)$). Thus, for $A_j\in \Pi$ such that $d_T(x, A_j)\ne d_T(y, A_j)$,  there exist $a,b\in A_j\cap V_i$   such that $d_T(x, A_j)=d_T(x,a)\ne d_T(y,b)=d_T(y, A_j)$. Hence, $d_G(x, A_j^G)=d_T(x, A_j)\ne d_T(y, A_j)=d_G(y, A_j^G)$.

Case 2.  $x\in V_i$, $y\in V_j$, $i\ne j$. We claim that there exists $r\in \left\{0,1,\left\lfloor\frac{k}{2}\right\rfloor\right\}$ such that $d_G(x,c_r)\ne d_G(y,c_r)$. We proceed by contradiction, i.e., suppose that $d_G(x,c_r)= d_G(y,c_r)$, for every $r\in \left\{0,1,\left\lfloor\frac{k}{2}\right\rfloor\right\}$. In this case we obtain the following three equalities.
\begin{equation}
d_G(x,c_i)+d_G(c_i,c_r)= d_G(y,c_j)+d_G(c_j,c_r),\quad r\in \left\{0,1,\left\lfloor\frac{k}{2}\right\rfloor\right\},
\end{equation}
or equivalently, \begin{equation}\label{igualdades-distancias}
d_G(x,c_i)-d_G(y,c_j)=d_G(c_j,c_r)-d_G(c_i,c_r),\quad r\in \left\{0,1,\left\lfloor\frac{k}{2}\right\rfloor\right\}.
\end{equation}
Now we differentiate the following subcases:
\\
\noindent
2.1. $1<i<j<\left\lfloor\frac{k}{2}\right\rfloor$. For $r=1$ and $r=\left\lfloor\frac{k}{2}\right\rfloor$ in (\ref{igualdades-distancias}) we deduce $j-i=i-j$, which is a contradiction.
\\
\noindent
2.2. $\left\lfloor\frac{k}{2}\right\rfloor<i<j<k$. For $r=1$ and $r=\left\lfloor\frac{k}{2}\right\rfloor$ in (\ref{igualdades-distancias}) we deduce $i-j=j-i$, which is a contradiction.
\\
\noindent
2.3. $1<i<\left\lfloor\frac{k}{2}\right\rfloor$ and $\left\lfloor\frac{k}{2}\right\rfloor<j<k$. For $r=0$ and $r=1$ in (\ref{igualdades-distancias}) we deduce $k-j-i=k-i-j+2$, which is a contradiction.

Therefore, $\Pi^G$ is a resolving partition for $G$.
\end{proof}


\begin{thebibliography}{99}

\bibitem{brigham} R. C. Brigham, G. Chartrand, R. D. Dutton, P.
Zhang, Resolving domination in graphs, {\it Mathematica Bohemica}
{\bf 128} (1) (2003) 25--36.

\bibitem{pelayo1} J. Caceres, C. Hernando, M. Mora, I. M. Pelayo, M. L. Puertas, C. Seara, D. R. Wood,
On the metric dimension of Cartesian product of graphs, {\it SIAM
Journal of Discrete Mathematics} {\bf 21} (2) (2007) 273--302.


\bibitem{pelayo2} J. Caceres, C. Hernando, M. Mora, I. M. Pelayo, M. L. Puertas, C. Seara, On the
metric dimension of some families of graphs, {\it Electronic Notes
in Discrete Mathematics} {\bf 22} (2005) 129--133.

\bibitem{chappell} G. Chappell, J. Gimbel, C. Hartman, Bounds on
the metric and partition dimensions of a graph, {\it Ars Combinatoria} {\bf 88} (2008) 349--366.

\bibitem{chartrand} G. Chartrand, L. Eroh, M. A. Johnson, O. R.
Oellermann, Resolvability in graphs and the metric dimension of a
graph, {\it Discrete Applied Mathematics} {\bf 105} (2000) 99--113.

\bibitem{chartrand1} G. Chartrand, C. Poisson, P. Zhang, Resolvability and the upper
dimension of graphs, {\it Computers and Mathematics with Applications} {\bf 39} (2000) 19--28.

\bibitem{chartrand2} G. Chartrand, E. Salehi, P. Zhang, The partition dimension of
a graph, {\it  Aequationes Mathematicae} (1-2) \textbf{59} (2000) 45--54.

\bibitem{fehr} M. Fehr, S. Gosselin, O. R. Oellermann, The
partition dimension of Cayley digraphs, {\it Aequationes
Mathematicae} {\bf 71} (2006) 1--18.

\bibitem{harary} F. Harary, R. A. Melter, On the metric dimension of a graph, {\it Ars Combinatoria} {\bf 2} (1976)
191--195.

\bibitem{haynes} T. W. Haynes, M. Henning, J. Howard, Locating
and total dominating sets in trees, {\it Discrete Applied
Mathematics} {\bf 154} (2006) 1293--1300.

\bibitem{robots} B. L. Hulme, A. W. Shiver, P. J. Slater, A Boolean algebraic analysis of fire protection,
{\it Algebraic and Combinatorial Methods in Operations Research}
{\bf 95} (1984) 215--227.

\bibitem{pharmacy1} M. A. Johnson, Structure-activity maps for visualizing the graph variables arising in drug
design, {\it Journal of Biopharmaceutical Statistics} {\bf 3} (1993) 203--236.

\bibitem{pharmacy2} M. A. Johnson, Browsable structure-activity datasets, {\it Advances in Molecular
Similarity} (R. Carb\'{o}-Dorca and P. Mezey, eds.) JAI Press
Connecticut (1998) 153--170.

\bibitem{landmarks} S. Khuller, B. Raghavachari, A. Rosenfeld, Landmarks in
graphs, {\it Discrete Applied Mathematics} {\bf 70} (1996) 217--229.

\bibitem{Tomescu1} R. A. Melter,  I. Tomescu, Metric bases in digital geometry, \emph{Computer Vision  Graphics and Image Processing} \textbf{25 }(1984) 113--121.

\bibitem{PartitionDimTrees}C. Poisson and P. Zhang, The metric dimension of unicyclic graphs, \emph{Journal of Combinatorial Mathematics and Combinatorial Computing} \textbf{40} (2002) 17--32.

\bibitem{Saenpholphat1}V. Saenpholphat and  P. Zhang,  Connected partition dimensions of graphs, \emph{Discussiones Mathematicae Graph Theory}  \textbf{22} (2)  (2002)  305--323.

\bibitem{survey} V. Saenpholphat and P. Zhang, Conditional resolvability in graphs: a
survey, {\it International Journal of Mathematics and Mathematical
Sciences} {\bf 38} (2004) 1997--2017.

\bibitem{leaves-trees} P. J. Slater, Leaves of trees, Proc. 6th Southeastern Conference on Combinatorics, Graph
Theory, and Computing, {\it Congressus Numerantium} {\bf 14} (1975)
549--559.

\bibitem{tomescu} I. Tomescu, Discrepancies between metric and
partition dimension of a connected graph, {\it Discrete Mathematics}
{\bf 308} (2008) 5026--5031.

\bibitem{partitionDimensionCorona} J. A. Rodríguez-Velázquez, I. G. Yero and D. Kuziak,  Partition dimension of corona product graphs. \emph{Ars Combinatoria}. To appear.

\bibitem{YRpd}
I. G. Yero and J. A. Rodr\'{\i}guez-Vel\'{a}zquez. A note on the partition dimension of Cartesian product graphs. \emph{Applied Mathematics and Computation} \textbf{217} (7) (2010) 3571--3574.

\bibitem{YKRdim} I. G. Yero, D. Kuziak and J. A. Rodr\'{\i}guez-Vel\'{a}zquez, On the metric dimension of corona product graphs. \emph{Computers and Mathematics with Applications} \textbf{61} (9) (2011) 2793--2798.
\end{thebibliography}
\end{document}